\documentclass[12pt]{article}
\usepackage{mathrsfs}
\usepackage{amsmath}
\usepackage{float}
\usepackage{cite}
\usepackage{esint}
\usepackage[all]{xy}
\usepackage{amssymb}
\usepackage{amsfonts}
\usepackage{amsthm}
\usepackage{color}
\usepackage{graphicx}
\usepackage{hyperref}
\usepackage{bm}
\usepackage{indentfirst}
\usepackage{geometry}
\theoremstyle{plain}
\newtheorem{thm}{Theorem}[section]

\newtheorem{defn}[thm]{Definition}
\newtheorem{prop}[thm]{Proposition}

\newtheorem{cor}[thm]{Corollary}
\newtheorem{lem}[thm]{Lemma}

\newtheorem{them}[thm]{Theorem}
\theoremstyle{remark}
\newtheorem{ex}[thm]{Example}
\newtheorem{rmk}[thm]{Remark}
\newenvironment{thmbis}[1]
  {\renewcommand{\thethm}{\ref{#1}}%
   \addtocounter{thm}{-1}%
   \begin{them}}
  {\end{them}}

\newcommand{\R}{\mathbb{R}}

\newcommand{\C}{\mathbb{C}}
\newcommand{\D}{\mathbb{D}}


\numberwithin{equation}{section}

\title{The existence criterion of holomorphic discs for higher $A_\infty$ operations via minimal discs}
\author{Qiang Tan and Zuyi Zhang}
\date{ }

\begin{document}

\maketitle
\begin{abstract}
    The main theorem of the paper provides an existence criterion of holomorphic discs for higher $A_\infty$ operations. The key step is to show that if a minimal disc in a K\"ahler manifold with boundary in a sequence of Lagrangian submanifolds intersecting transversely such that its partial Maslov indices are either all no less than $1$ or all no larger than $-1$, then there is a holomorphic disc with the same image as this minimal disc. As a by-product, we show that all minimal discs in $\C\mathrm{P}^m$ with boundary on $\R\mathrm{P}^m$ are holomorphic.
\end{abstract}

\section{Introduction}

This paper is motivated by the $A_\infty$ category of symplectic manifolds arising from Lagrangian Floer theory. This category is equipped with a sequence of operations $m_l$, $l=1,2,3,\ldots$, satisfying the $A_\infty$ relations \cite{fukaya2009lagrangian}, where the objects of this category are unobstructed Lagrangian submanifolds. These operations are defined by counting elements of zero-dimensional moduli spaces obtained by solving the Cauchy–Riemann equation. When $l\ge2$, the moduli space consists of pseudo-holomorphic discs connecting intersection points of the given Lagrangians (Figure \ref{fig:mudisc}). And in this situation, the Maslov index of the discs is less or equal than $-2$. The main result of this paper gives a criterion for the existence of such discs on closed K\"ahler manifolds using partial Maslov indices.

\begin{them}\label{thm:main}
    Given a closed K\"ahler manifold $(X,\omega,J)$ and a sequence of Lagrangians $L_k$ with $k=1,\dots,n$ intersecting transversely in $X$, suppose there is a smooth disc $u$ connecting $\vec x=(x_1,\ldots,x_{n-1})$ and $y\in L_1\cap L_n$ with $x_k\in L_k\cap L_{k+1}$ for $k=1,\ldots,n-1$ $($Definition \ref{def:holodisc}$)$. Then there is a minimal disc $u_0$ connecting $\vec x$ and $y$ homotopic to $u$. If in addition, the partial Maslov indices $($Definition \ref{def:parmas}$)$ of $u_0$ are either all no less than $1$ or all no larger than $-1$, then there is a holomorphic map from $\D$ to $X$ with the same image as $u_0$.
\end{them}

\begin{rmk}
    On K\"ahler surfaces (real dimension 4), there are two partial Maslov indices for $u_0$. Therefore the Maslov index of $u_0$ is either greater than 2 or less than $-2$ in case the above theorem can be applied. It implies that the criterion only applies to $m_l,\ l\ge2$, for K\"ahler surface. Similarly, given a K\"ahler manifold with real dimension $2n$, the criterion above is applicable for $m_l,\ l\ge n$. \textbf{This is the reason why these holomorphic discs can correspond only to higher $A_\infty$ operations.}
\end{rmk}

Oh \cite{oh1995riemann} used the Brikhoff factorization to show that a Lagrangian subbundle of $E|_{\partial\D}$ has a standard form that is characterized by the partial Maslov indices, where $E$ is a holomorphic vector bundle over the unit disc $\D$. Unlike the Maslov index, partial Maslov indices are not topological and hard to compute. However, the positivity of the partial Maslov indices is equivalent to the Griffiths positivity (Definition \ref{def:grfpos}) of the double of the pullback bundle through $u_0$.

\begin{them}\label{thm:main2}
    Suppose $L_k$ with $k=1,\dots,n$ is a sequence of Lagrangians in a K\"ahler manifold $(X^m,\omega, J)$ intersecting transversely. Let $u:\D\rightarrow X$ be a disc connecting the intersections $\vec x,y$ of $\{L_k\}_k$. Let $E^D$ be the doubled holomorphic bundle defined as in Definition \ref{def:doubundle}, where $E=u^*T^{1,0}X$ and the boundary Lagrangian subbundle $F_k$ is $u^*T^{1,0}L_k|_{u(\partial\D)\cap L_k}$ $($Equation \ref{equ:10lagtan}$)$ for $k=1,\ldots,m$. Then the partial Maslov indices of $u$ are all larger than $1$ if and only if $E^D$ is Griffiths positive. $($The bundle $E$ is holomorphic by Lemma \ref{lem:holvec}.$)$
\end{them}

As a by-product, we can describe all minimal discs in $\C\mathrm{P}^m$ with boundary in $\R\mathrm{P}^m$. The proof is presented in Section \ref{sec:3}

\begin{cor}\label{cor:minidisccpm}
    Suppose $\C\mathrm{P}^m$ is equipped with the Fubini-Study metric. Then all minimal discs in $\C\mathrm{P}^m$ with boundary in $\R\mathrm{P}^m$ are all holomorphic. As a result, these discs are contained in a $\C\mathrm{P}^1\subset \C\mathrm{P}^m$.
\end{cor}

The use of pseudo-holomorphic curves to study symplectic manifold was introduced by Gromov \cite{gromov1985pseudo}. Then Floer \cite{floer1988morse} defined the Lagrangian Floer homology to solve a version of Arnold's conjecture, where the differential of this homology is defined by counting the number of pseudo-holomorphic discs. Fukaya discovered the $A_\infty$ structure in Lagrangian Floer theory and it was later called the Fukaya category. However, Lagrangian Floer homology and Fukaya category may not be well defined in general because of the bubbling phenomena. Fukaya-Oh-Ohto-Ono \cite{fukaya2009lagrangian} used algebraic methods to find out the obstructions to defining the $A_\infty$ structures. There has been extensive research on the Fukaya category. Abouzaid \cite{abouzaid2008fukaya} gives a thorough description of the Fukaya categroy on higher-genus surfaces. Hedden-Herald-Hogancamp-Kirk \cite{hedden2020fukaya} calculate examples for the pillowcase. Cazassus provides a construction of equivariant Lagrangian Floer homology using the Morse model \cite{cazassus2024equivariant}. The second author gives a combinatorial way for computing the boundary maps in Lagrangian Floer theory when the symplectic manifold is a Cartesian product of closed surfaces \cite{zhang2024construction,zhang2025uniqueness}. The references on the Fukaya category include Auroux \cite{auroux2014beginner} and Seidel \cite{seidel2008fukaya}.\\


For another point of view, there are many studies on constructing pseudo-holomorphic curves. Gromov showed the existence of pseudo-holomorphic curves by analyzing the $\bar\partial$ operator directly \cite{gromov1985pseudo}. Siu-Yau \cite{siu1980compact}, Micallef \cite{micallef1982stable}, and Arezzo \cite{arezzo1998stable} proved the existence of holomorphic curves using stable minimal surfaces. Recently, Chen \cite{chen2025holomorphic} tried to establish the existence of holomorphic curves in an open K\"ahler surface with its boundary on a given Lagrangian submanifold. This work is based on the existence of minimal surface of Ye \cite{ye1991existence}.

We use the idea from Chen \cite{chen2025holomorphic} and Chen-Fraser \cite{chen2015minimal}. Compared with \cite{chen2025holomorphic} \cite{chen2015minimal}, the ambient manifold in the present paper is a closed K\"ahler manifold. So the existence of the minimal surface is relatively straightforward. The boundary of the holomorphic curve in the current paper lies on $n$ Lagrangian submanifolds, where the $n$-th unit roots are mapped to prescribed intersection points of these Lagrangians (Figure \ref{fig:mudisc}). The method used in \cite{chen2025holomorphic} \cite{chen2015minimal} requires the kernel of the linearized Cauchy-Riemann operator to have dimension at least 2, and therefore imposes restrictions on the partial Maslov indices.\\

The arrangement of this paper is as follows. In Section \ref{sec:1}, we provide the necessary background of symplectic geometry. Section \ref{sec:2} aims at proving the existence of minimal surface with the prescribed Lagrangian boundary condition. Section \ref{sec:partms} defines the partial Maslov indices and Theorem \ref{thm:main2} is proved. The proof of Theorem \ref{thm:main} and Corollary \ref{cor:minidisccpm} is given in Section \ref{sec:3}.\\

\noindent\textit{Acknowledgements:} The second author would like to thank Jingyi Chen for enlightening conversations, as well as Jie Zhou and Kenji Fukaya for helpful suggestions.

\section{Preliminary}\label{sec:1}

The goal of this section is to provide the necessary background on holomorphic discs in K\"ahler manifolds whose boundaries lie on given Lagrangian submanifolds. These Lagrangians are assumed to intersect transversely. After defining the Cauchy-Riemann operator, we give the definition of the Maslov index.  For generic complex structures on K\"ahler manifolds, the moduli space of holomorphic discs (Definition \ref{def:holodisc}) is generically a smooth manifold whose dimension can be computed from the Maslov index and the number of marked points (corresponding the unit roots in Definition \ref{def:holodisc}) by the Riemann-Roch theorem \cite[Equation~2.5]{auroux2014beginner}. 

\begin{defn}
    A \textbf{K\"ahler manifold} is a triple $(X,\omega, J)$, where $X$ is a closed manifold equipped with a complex structure $J$ and a symplectic form $\omega$ such that $\omega(\cdot,J\cdot)>0$ and $\omega(J\cdot,J\cdot)=\omega(\cdot,\cdot)$.
\end{defn}

In Lagrangian Floer theory, the Lagrangian submanifolds should intersect transversely.

\begin{defn}
    Given a symplectic manifold $(X, \omega)$, a \textbf{Lagrangian} of $X$ is an embedded submanifold $l:L\rightarrow X$ such that $l^*\omega=0$.
\end{defn}

\begin{defn}
    Given two submanifolds $N_1$ and $N_2$ of a smooth manifold $X$, $x\in N_1\cap N_2$ is said to be a \textbf{transversal intersection} if $T_xN_1+T_xN_2=T_xX$. If all the intersections of $N_1$ and $N_2$ are transversal intersections, then we say that $N_1$ and $N_2$ \textbf{intersect transversely}.
\end{defn}

\begin{ex}
    Suppose $L_1$ and $L_2$ are two Lagrangians in a K\"ahler surface $(X,\omega, J)$ intersecting transversely. Note that $\dim L_1=\dim N_2=2$, then $L_1$ and $L_2$ intersecting at $x\in L_1\cap L_2$ if and only if $T_xL_1\cap T_xL_2=\{0\}$ by dimension counting.
\end{ex}

\begin{defn}
    Suppose $(X, \omega, J)$ is a K\"ahler manifold. Let $u:\D\rightarrow X$ be a smooth map from the unit disc $\D\subset \C$. The \textbf{Cauchy-Riemann operator} is defined as 
    \[
    \bar\partial_J(u):=\frac{1}{2}(du+J\circ du\circ j),
    \]
    where $j$ is the standard complex structure inherent from $\C$.
\end{defn} 

Now we are ready to define holomorphic discs connecting the intersections of given Lagrangians.

\begin{defn}\label{def:holodisc}
    Suppose $L_k$ with $k=1,\dots,n$ is a sequence of Lagrangians in a K\"ahler manifold $(X,\omega, J)$ intersecting transversely. Denote $\D$ as the unit disc in $\C$. Let $x_k\in L_k\cap L_{k+1}$ and $y\in L_1\cap L_n$ be intersection points with $k=1,\ldots,n-1$. \textbf{A disc $u$ connecting $\vec{x}:=(x_1,\ldots,x_{n-1})$ and $y$} is a smooth map $u:(\D,\partial\D)\rightarrow (X,\cup_k L_k)$ such that 
    \begin{itemize}
        \item $u(e^{\frac{2\pi k}{n}i})=x_k$ with $k=1,\ldots,n-1$, $u(1)=y$, and $u(I_k)\subset L_k$ with $k=1,\ldots,n$. Here $I_k$ stands for the circle segment in $\partial\D$ between the complex unit roots $e^{\frac{2\pi (k-1)}{n}i}$ and $e^{\frac{2\pi k}{n}i}$ $($Figure \ref{fig:mudisc}$)$.
    \end{itemize}
    If moreover 
    \begin{itemize}
        \item $\bar\partial_Ju=0$,
    \end{itemize}
    then $u$ is a \textbf{holomorphic disc connecting $\vec x$ and $y$}. The moduli space of all such holomorphic discs is denoted as $\mathcal{M}(\vec x,y)$.
\end{defn}

\begin{rmk}
    In fact, the unit roots chosen here are not essential. We can replace the unit roots by any $n$ ordered marked points.
\end{rmk}

\begin{figure}[H] 
\centering 
\includegraphics[width=1\textwidth]{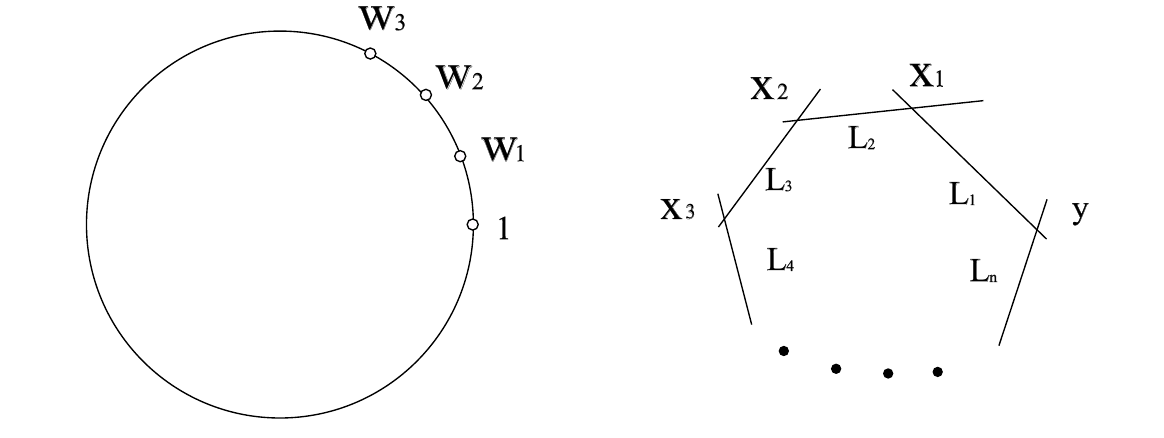}
\caption{In the picture, $w_k=e^{\frac{2\pi k}{n}i}$ is mapped to $x_k$ and $1$ is mapped to $y$.}
\label{fig:mudisc}
\end{figure}

To better understand the Cauchy-Riemann operator, the following description is needed. Let $TX\otimes\C$ be the complexified tangent bundle of a complex manifold $(X,J)$ with complex dimension $m$. The complex structure $J$ extends $\C$-linearly to $TX\otimes\C$. Since $J^2=-1$, $TX\otimes\C$ splits as the direct sum of $T^{0,1}X$ and $ T^{1,0}X$ by the $J$ action, where $T^{0,1}X$ and $ T^{1,0}X$ are the eigenspaces of $J$ with eigenvalues $-i$ and $i$ respectively.

Given a smooth map $u=(u_1,\ldots,u_m):\D\rightarrow X$ from the unit disc in $\C$, denote $E$ as the pull-back bundle $u^*T^{1,0}X$. Taking $s+it$ as the coordinate on $\D$, the Cauchy-Riemann operator can be viewed as an $E$-valued $(0,1)$-form:
\[
\bar\partial_J u_k=(\frac{1}{2}(\partial_su_k+J(u)\partial_tu_k)ds+\frac{1}{2}(\partial_tu_k+J(u)\partial_su_k)dt)\quad k=1,2,\ldots,m.
\]

The next goal is to define the Maslov index. This index is related to the Fredholm index of the linearized Cauchy-Riemann operator and the dimension of the moduli space of holomorphic discs. More details can be found in McDuff-Salamon \cite{mcduff2017introduction} \cite{mcduff2025j} and Auroux \cite{auroux2014beginner}.

Let $(\C^m,\omega_{std})$ be the symplectic manifold equipped with the standard symplectic form on $\C$. Denote $\Lambda(m)$ as the Lagrangian Grassmannian of $\C^m$, that is, the space of Lagrangian subspaces of $(\C,\omega_{std})$. In fact, $\Lambda(m)\cong \mathrm{U(m)/O(m)}$ \cite{mcduff2017introduction}.

\begin{defn}
    Let $\tau:S^1\rightarrow\Lambda(m)$ be a smooth loop in $\Lambda(m)$. The \textbf{Masolv index $\mathbf{\mu(c)}$ of $\mathbf{\tau}$} is defined as the topological degree of the composition of the following map:
    \[
    S^1\xrightarrow{\tau}\Lambda(m)\cong\mathrm{U(m)/O(m)}\xrightarrow{\det^2} S^1,
    \]
    where the last arrow stands for mapping a class $[a]\in\mathrm{U(m)/O(m)}$ to $\det(a)^2$ $($ clearly this map is independent of the choice representing the class$)$.
\end{defn}

The Maslov index needed in this paper is slightly different from the one described above and is explained in the following. Given a symplectic manifold $(X,\omega)$ and a sequence of Lagrangians $L_k$ with $k=1,\dots,n$ intersecting transversely in $X$. Let $u:\D\rightarrow X$ be a disc connecting $\vec x=(x_1,\ldots,x_{n-1})$ and $y$, as in Definition \ref{def:holodisc}. However, the Maslov index of the boundary loop $u|_{\partial\D}$ is not defined, since this loop is not smooth at $\vec x$ and $y$. By the Darboux theorem, there is a neighborhood $U$ of each point $x_k $ and $ y$ that is symplectomorphic to $(\C^m,\omega_{std})$, such that the images of $L_k$ and $L_{k+1}$ (set $L_{n+1}=L_1$) are $\R^m$ and $i\R^m$, respectively. Viewing the Lagrangian subspaces $\R^m$ and $i\R^m$ as two points in $\Lambda(m)$, there is a \textbf{canonical short path} connecting them (c.f. \cite{auroux2014beginner}). It is given by
\begin{align*}
    \lambda_{x_k}:[0,1]&\rightarrow\Lambda(m)\\
    t&\mapsto e^{-\frac{\pi i}{2}t}\R^m.
\end{align*}
To define the Maslov index of $u$, observe that the tangent space of ${u(z)}\in \cup_kL_k$ can be regarded as a point in $\Lambda(m)$ when $z\in \partial\D-\{e^{\frac{2\pi k}{n}i}\}_k$. Therefore, $u|_{I_k}$ induces a map, still denote by $u|_{I_k}$, from $I_k$ to $\Lambda(m)$, where $\{I_k\}_{k=1}^n$ stands for the connected components of $\partial\D-\{e^{\frac{2\pi k}{n}i}\}$ ordered counterclockwise starting at 1. Combining this with the canonical short path, we can define the Maslov index of $u$.
\let\oldthethm\thethm
\begin{defn}[\cite{auroux2014beginner}]
    Suppose $L_k$ with $k=1,\dots,n$ is  a sequence of Lagrangians in a K\"ahler manifold $(X,\omega, J)$ intersecting transversely. The canonical short paths at $x_k$ and $y$ are denoted as $\lambda_{x_k}$ and $\lambda_y$, respectively. Let $u:\D\rightarrow X$ be a disc connecting the intersections $\vec x=(x_1,\ldots,x_{n-1}),y$ of $\{L_k\}_k$, then the \textbf{Maslov index of $\mathbf{u}$} is defined as
    \[
    \mu(u):=\mu( u|_{I_1}^{-1}\circ\lambda_{x_1}\circ\ldots\circ u|_{I_{n-1}}^{-1}\circ\lambda_{x_{n-1}}\circ u|_{I_n}^{-1}\circ\lambda_y),
    \]
    where $u|_{I_k}^{-1}$ stands for travelling along $I_k$ counterclockwisely.
\end{defn}

Since the Maslov index is a topological invariant \cite{mcduff2017introduction} \cite{robbin1993maslov}, $u|_{I_k}^{-1}\circ\lambda_{x_k}$ can be isotoped to $v_k$ for $k=1,\ldots,n$. Here $v_k$ is defined as multiplying $u|_{I_k}^{-1}(t)\subset\C^m$ by the diagonal matrix $\Phi_k(t)=\begin{bmatrix}
    e^{-\frac{\pi i}{2|I_k|}t}&0&\ldots&0\\
    0&e^{-\frac{\pi i}{2|I_k|}t}&\ldots&0\\
    ..&..&..&..\\
    0&0&\ldots&e^{-\frac{\pi i}{2|I_k|}t}
\end{bmatrix}\in\mathrm{GL}(\C,m)$ with $t\in I_k$. Therefore
\[
\mu( u|_{I_1}^{-1}\circ\lambda_{x_1}\circ\ldots\circ u|_{I_{n-1}}^{-1}\circ\lambda_{x_{n-1}}\circ u|_{I_n}^{-1}\circ\lambda_y)=\mu(v_1\circ\ldots\circ v_n),
\]
and we have an equivalent definition below. This definition is easier to use when defining the partial Maslov index, which is essential to determine the existence of holomorphic curves.

\addtocounter{thm}{-1}
\renewcommand{\thethm}{\oldthethm$'$}
\begin{defn}\label{def:msl2}
    Suppose $L_k$ with $k=1,\dots,n$ is  a sequence of Lagrangians in a K\"ahler manifold $(X,\omega, J)$ intersecting transversely. Let $u:\D\rightarrow X$ be a disc connecting the intersections $\vec x,y$ of $\{L_k\}_k$, then the \textbf{Maslov index of $\mathbf{u}$} is defined as
    \[
    \mu(u):=\mu(v_1\circ\ldots\circ v_n).
    \]
\end{defn}
\renewcommand{\thethm}{\oldthethm}

\begin{rmk}
    More details of the Maslov index can be found in Auroux \cite{auroux2014beginner}, Chan-Suen \cite{chan2019syz}, Mcduff-Salamon \cite{mcduff2017introduction}, and Robbin-Salamon \cite{robbin1993maslov}.
\end{rmk}

The Maslov index provides the virtual dimension of holomorphic discs.
\begin{them}[{\cite[Equation~2.5]{auroux2014beginner}}]
    Suppose $L_k$ with $k=1,\dots,n$ is  a sequence of Lagrangians in a K\"ahler manifold $(X,\omega, J)$ intersecting transversely. The virtual dimension $\dim\mathcal{M}(\vec x,y)$ of the holomorphic discs $u$ connecting $\vec x$ and $y$ as in Definition \ref{def:holodisc} is 
    \[
    \dim\mathcal{M}(\vec x,y)=n-3+\mathrm{ind}(D_u\bar\partial_J)=n-3+\mu(u),
    \]
    where $D_u\bar\partial_J$ is the first variation of $\bar\partial_J$ at $u$.
\end{them}

The following gives an explicit description of $D_u\bar\partial_J$. Since $X$ is a K\"ahler manifold, the complex structure on $TX$ is compatible with the Levi-Civita $\nabla$ connection. Moreover, the $(0,1)$-part of $\nabla$ is the Dolbeault operator $\bar\partial_X$. Before proceeding, the following lemma is needed.

\begin{lem}\label{lem:holvec}
    Suppose $u:\D\hookrightarrow X$ is a smooth map mapping into the K\"ahler manifold $(X,\omega,J)$. After pulling back $\nabla$ to the bundle $E:=u^*T^{1,0}X$ to get a connection $\nabla_E$, then the operator $\bar\partial_E$ defined by taking the $(0,1)$-part of $\nabla_E$ defines a holomorphic structure on $E$.
\end{lem}

\begin{proof}
    After equipping $\D$ with the standard complex structure, the complex vector bundle $E$ is a holomorphic bundle if and only if $\bar\partial_E^2=0$ by \cite[Theorem~2.1.53]{donaldson1997geometry}. Recall that there are no non-trivial $(0,2)$-forms over $\D$. As a result,
    \[
    \bar\partial_E^2=F^{(0,2)}_{\nabla_E}=0.
    \]
\end{proof}

Combining the above lemma and McDuff-Salamon \cite[Equation~3.1.6]{mcduff2025j}, we have
\begin{equation}\label{equ:varcroper}
    D_u\bar\partial_J=\nabla^{(0,1)}_E=\bar\partial_E.
\end{equation}

\section{Minimal surface method}\label{sec:2}

The existence of minimal manifolds in closed K\"ahler manifolds is proved in this section. This is a necessary step in the proof of the main theorem. We first define the admissible set and then use the variational method in Evans \cite{evans} to show the existence of a minimal disc. For technical reasons, the proof that the minimal surface is perpendicular to the Lagrangians is postponed to the final part of this section.

\begin{defn}
    Given a K\"ahler manifold $(X,\omega,J)$ and a sequence of Lagrangians $L_k$ with $k=1,\dots,n$ intersecting transversely in $X$, the \textbf{Dirichlet integral} of a smooth map $u:(\D,\partial\D)\rightarrow(X,\cup_k L_k)$ is defined as
    \[
    \mathcal{D}[u]:=\frac12\int_\D g(\partial_su,\partial_su)+g(\partial_tu,\partial_tu)dvol_\D,
    \]
    where $g(\cdot,\cdot):=\omega(\cdot,J\cdot)$ is a Riemann metric and $s+it$ is the standard complex coordinates on $\D$.
\end{defn}

Recall that for any smooth manifold $X$ and a sequence of Lagrangians $L_k$ with $k=1,\dots,n$, there is a relative homotopy long exact sequence:
\[
\cdots\rightarrow\pi_2(X)\rightarrow\pi_2(X,\cup_k L_k)\xrightarrow{\partial}\pi_1(\cup_k L_k)\xrightarrow{i_*}\pi_1(X)\cdots.
\]

\begin{defn}\label{def:admset}
    Given a symplectic manifold $(X,\omega)$ and a sequence of Lagrangians $L_k$ with $k=1,\dots,n$ intersecting transversely in $X$, the \textbf{admissible set} $\mathcal{P}_{\vec x,y}(\gamma,A)$ is the $W^{1,2}$ completion of the following set
    \[
    \{u\ is\ a\ disc\ connecting\ \vec x,y|\ [u]=A,\ [u|_{\partial \D}]=[\gamma]\}\ \ \mathrm{(Definition}\ \ref{def:holodisc}\mathrm{)},
    \]
    where $A\in\pi_2(X,\cup_kL_k)$ and $\gamma:S^1\rightarrow \cup_kL_k$ is a loop such that each connected component $I_k$ of $S^1-\{e^{\frac{2\pi k}{n}i}\}$ is mapped to $L_k$. Here $I_k$ is ordered counterclockwisely starting at $1\in S^1$. Moreover, $\partial A=[\gamma]\ne0$ and $i_*[\gamma]=0$ in the above long exact sequence.
\end{defn}

\begin{defn}
    Given a symplectic manifold $(X,\omega)$ and a sequence of Lagrangians $L_k$ with $k=1,\dots,n$ intersecting transversely in $X$, a disc $u$ connecting $\vec x,y\in L_1\cap L_2$ is called a \textbf{minimal surface} if
    \[
    \mathcal{D}[u]=\inf_{v\in \mathcal{P}_{\vec x,y}(\gamma,A)}\mathcal{D}[v].
    \]
\end{defn}

The next goal of this section is to study the existence of minimizers.
The approach we use follows the proof of the existence of minimizers in Section $8.2$ of \cite{evans}.
We first recall some details from \cite{evans}.
In Section $8.2$ of \cite{evans}, Evans defines the coercivity and
convexity of the function $L$, which ensure that the functional
$$
I[w]:=\int_{U}L(Dw(x), w(x), x)dx
$$
has a minimizer.
The coercivity condition is  
\begin{equation}
  \left\{
    \begin{array}{ll}
  {\rm there \,\,\,exist\,\,\, constants\,\,\, \alpha>0,\,\,\,\beta\geq 0\,\,\,such\,\,\,that} & \\
     &   \\
    L(p,z, x )\geq \alpha|p|^q-\beta &   \\
       &   \\
   {\rm for\,\,\, all }\,\,\,p=(p_1\ldots,p_n)\in\R^m,\,\,\,z\in\R,\,\,\,x\in U.
   \end{array}
  \right.
  \end{equation}
  Therefore,
  $$
  I[w]\geq\alpha\|Dw\|^q_{L^q(U)}-\gamma
  $$
  for $\gamma:=\beta\int_{U}dx$.
  The coercivity condition ensures that the minimizing sequence
 is bounded in $W^{1,q}(U)$, and hence admits a weakly convergent subsequence.
 Then, combined with the convexity condition of integrand $L$,
 $$
 \sum_{i,j=1}^n\frac{\partial^2L}{\partial p_i\partial p_j}(Dw(x), w(x), x)\xi_i\xi_j\geq 0\,\,\,(\xi\in\R^m,\,\,\,x\in U),
 $$
  this leads to
the weak lower semicontinuity
$$
I[w]\leq \liminf_{k\rightarrow\infty}I[w_k]
$$
 whenever 
 $$
 w_k\rightharpoonup w\,\,\, {\rm weakly\,\,\,in}\,\,\, W^{1,q}(U).
 $$
Finally, Evans proves the existence of minimizer.
\begin{them}
[\cite{evans} Theorem $2$, Section $8.2$]
Assume that $L$ satisfies the coercivity inequality and is convex in the variable $p$.
Suppose also the admissible set $\mathcal{A}$ is nonempty.
Then there exists at least one function $\tilde{w}\in\mathcal{A}$ solving
$$
I[\tilde{w}]=\min_{w\in\mathcal{A}}I[w].
$$
\end{them}
For more details, we refer to Section $8.2$ in \cite{evans}. In our context, the ambient manifold is compact, and the integrand is $g(\partial_su,\partial_su)+g(\partial_tu,\partial_tu)$, which naturally satisfy the coercivity and convexity conditions for $q=2$ specified there. Thus together with standard regularity theory, we can obtain the following result.
\begin{them}\label{thm:minisurf}
    Given a K\"ahler manifold $(X,\omega,J)$ and a sequence of Lagrangians $L_k$ with $k=1,\dots,n$ intersecting transversely in $X$, then there exists a smooth minimal surface in the admissible set $\mathcal{P}_{\vec x,y}(\gamma,A)$, where $\vec x,y$ are defined as in Definition \ref{def:holodisc}.
\end{them}

The rest of this section is to show that the disc we find in the above theorem is perpendicular to the given Lagrangians. This is a technical part in the proof of the main theorem. To achieve this, we need to show that the disc we find in the above theorem is in fact area minimizing.

\begin{defn}
    Given a K\"ahler manifold $(X,\omega,J)$ and a sequence of Lagrangians $L_k$ with $k=1,\dots,n$ intersecting transversely in $X$, the \textbf{area functional} of a smooth map $u:(\D,\partial\D)\rightarrow(X,\cup_kL_k)$ is defined as
    \[
    \mathcal{A}[u]:=\int_\D \sqrt{\det(Du^TDu)}\,dvol_\D,
    \]
    where $Du$ is the Jacobian matrix.
\end{defn}

When the minimizer of the Dirichlet integral is conformal, this element also minimizes the area functional. The references include Robertson \cite{robertson2020global} and Schmidt \cite{schmidtlecture}.
\begin{lem}[Robertson \cite{robertson2020global}]
    Let $(X,\omega,J)$ be a K\"ahler manifold and $g(\cdot,\cdot)=\omega(\cdot,J\cdot)$ is a Riemann metric. Then for any $u(s,t)$ in the admissible set $\mathcal{P}_{\vec x,y}(\gamma,A)$ $($Definition \ref{def:admset}$)$, we have
    \[
    \mathcal{A}[u]\le\mathcal{D}[u],
    \]
    where the equality holds if and only if $u$ is conformal.
\end{lem}


The minimizer of the Dirichlet integral defined at the beginning of this section is conformal, therefore the following is true.
\begin{them}\label{thm:areamini}
    Let $u_0$ be the minimal surface found in Theorem \ref{thm:minisurf}. Then $u$ minimizes the area functional as well since $u$ is conformal, i.e.
    \[
    \mathcal{A}[u_0]=\inf_{v\in \mathcal{P}_{\vec x,y}(\gamma,A)}\mathcal{A}[u].
    \]
\end{them}


%

As a corollary of Theorem \ref{thm:areamini}, the minimizer $u_0$ of the area functional $\mathcal{A}$ satisfies the perpendicular condition, that is, the outer normal vector of $u_0$ along the boundary is perpendicular to the tangent space of $L_k$ for $k=1,\ldots,n$. This is a technical point needed for the next section, as stated below.

\begin{cor}\label{cor:perpbd}
    Given a K\"ahler manifold $(X,\omega,J)$ and a sequence of Lagrangians $L_k$ with $k=1,\dots,n$ intersecting transversely in $X$, suppose $u_0$ is the minimizer of the area functional $\mathcal{A}$ among the admissible set $\mathcal{P}_{\vec x,y}(\gamma,A)$ $($Definition \ref{def:admset}$)$. Then the outer normal vector of $u_0$ along the boundary is perpendicular to the tangent space of $L_k$ when the image of $u$ from the boundary is in $L_k$ for $k=1,\ldots,n$.
\end{cor}

\begin{proof}
    According to Theorem 1 of White \cite{white2013lectures}, the first variation of the area functional is
    \[
    \frac{d}{dt}|_{t=0}\mathcal{A}(u_t)=\int_{\partial\D}V\cdot\nu_{\partial\D}-\int_\D H\cdot Vdvol_\D,
    \]
    where $u_t$ is a path in the admissible set $\mathcal{P}_{\vec x,y}(\gamma,A)$ starting at $u_0$, $V=\frac{du}{dt}$, $\nu_{\partial\D}$ is the outer normal vector of $\partial\D\subset\D$, and $H$ is the mean curvature. By the definition of $V$, $V|_{\partial\D}$ is tangent to the Lagrangian where $\partial\D$ is moving on. Because $u_1$ is the minimizer of $\mathcal{A}$, $\int_{\partial\D}V\cdot\nu_{\partial\D}=0$ for all such $V$. The corollary follows directly.
\end{proof}

\section{Partial Maslov indices for union of Lagrangians}\label{sec:partms}

The partial Maslov indices for loops in a single Lagrangian submanifold were defined by Oh \cite{oh1995riemann} using a result from Globevnik \cite{globevnik1994perturbation}. Evans-Lekili \cite{evans2015floer} later showed that the partial Maslov indices provide information on whether the linearization of the Cauchy-Riemann operator is surjective (i.e. $J$ is regular) or not. In this section, we first introduce the theorem in \cite{globevnik1994perturbation}, then generalize the definition of partial Maslov indices to the case where the loop lies in the union of several Lagrangian submanifolds.

\begin{prop}[Birkhoff factorization {\cite[Lemma~5.1]{globevnik1994perturbation}}]
    Let $A:\partial\D\rightarrow\mathrm{GL}(m,\C)$ be a map of class $\mathcal{C}^\alpha$, $0<\alpha<1$. Then there is a map $\Theta:\D\rightarrow\mathrm{GL}(m,\C)$ of class $\mathcal{C}^\alpha$ on $\partial\D$, holomorphic on $\D$ such that
    \[
    A(\zeta)\overline{A(\zeta)^{-1}}=\Theta(\zeta)\Lambda(\zeta)\overline{\Theta(\zeta)^{-1}}\ \ (\zeta\in\partial\D),
    \]
    where
    \[
    \Lambda(\zeta)=\begin{bmatrix}
        \zeta^{\kappa_1}&0&\ldots&0\\
        0&\zeta^{\kappa_2}&\ldots&0\\
        ..&..&..&..\\
        0&0&\ldots&\zeta^{\kappa_m}
    \end{bmatrix},
    \]
    with unique $\kappa_1\ge\kappa_2\ge\ldots\ge\kappa_m$.
\end{prop}

\begin{rmk}
    This theorem is based on solving a Riemann-Hilbert problem for vector functions proved by \cite{vekua1967systems}.
\end{rmk}

Since $\mathrm{U}(m)$ is a subgroup of $\mathrm{GL}(m,\C)$ and $\mathrm{U}(m)/\mathrm{O}(m)$ is identified with the Lagrangian Grassmannian space of $\C^m$ (\cite{mcduff2017introduction}), so the following is true.

\begin{cor}[{\cite[Theorem~4.7]{oh1995riemann}}]\label{cor:ohsplitting}
    Let $\tau:\partial\D\rightarrow\mathrm{U}(m)/\mathrm{O}(m)$ be a map of class $\mathcal{C}^\alpha$, $0<\alpha<1$. Then there is a map $\Theta:\D\rightarrow\mathrm{GL}(m,\C)$ of class $\mathcal{C}^\alpha$ on $\partial\D$, holomorphic on $\D$ such that
    \[
    \tau(\zeta)=\Theta(\zeta)\Lambda^\frac12(\zeta)\R^m\ \ (\zeta\in\partial\D),
    \]
    where
    \[
    \Lambda(\zeta)=\begin{bmatrix}
        \zeta^{\kappa_1}&0&\ldots&0\\
        0&\zeta^{\kappa_2}&\ldots&0\\
        ..&..&..&..\\
        0&0&\ldots&\zeta^{\kappa_m}
    \end{bmatrix},
    \]
    with $\kappa_1\ge\kappa_2\ge\ldots\ge\kappa_m$. The numbers $\kappa_1,\ldots,\kappa_m$ are defined as the \textbf{partial Maslov indices} and $\sum_k\kappa_k=\mu(\tau)$.
\end{cor}

\begin{rmk}
    \begin{itemize}
        \item[1.] The entry in matrix $\Lambda^\frac12(\zeta)$ is not well-defined for $\kappa_k$ being odd since $\zeta^{\frac{\kappa_k}{2}}$ is double valued. However, the two branches are differed by multiplying $-1$. Therefore, $\Lambda^\frac12(\zeta)\R^m$ is well-defined in the Lagrangian Grassmannian space.
        \item[2.] In \cite{oh1995riemann}, Oh proved the above corollary by assuming $A$ is a smooth loop. But $A$ only needs to be $\mathcal{C}^\alpha$, $0<\alpha<1$.
    \end{itemize}
\end{rmk}

In this paper, the disc $\D$ is in a K\"ahler manifold with boundary $\partial\D$ in a union of transversely intersecting Lagrangian submanifolds. Therefore, $\partial\D$ induces a loop in $\mathrm{U}(m)/\mathrm{O}(m)$ with discontinuities at the intersection points of the Lagrangians. To define the partial Masolv indices for this loop, we replace it by the loop in Definition \ref{def:msl2} and the resulting loop is piecewise smooth.

\begin{defn}\label{def:parmas}
    Suppose $L_k$ with $k=1,\dots,n$ is a sequence of Lagrangians in a K\"ahler manifold $(X,\omega, J)$ intersecting transversely. Let $u:\D\rightarrow X$ be a disc connecting the intersections $\vec x,y$ of $\{L_k\}_k$, then the \textbf{partial Maslov index of $\mathbf{u}$} is defined as the partial Maslov indices of $v_1\circ\ldots\circ v_n$ $(v_1\circ\ldots\circ v_n$ is defined in Definition \ref{def:msl2}$)$.
\end{defn}

In fact, by pulling back $T^{1,0}X$ to $\D$ through $u$ (in the above definition), one obtains a holomorphic vector bundle $E$ according to Lemma \ref{lem:holvec}. On the boundary $\partial\D$, we get a sequence of Lagrangian subbundles $F_k$ defined as the pullbacks of $T^{1,0}L_k|_{u(\partial\D)\cap L_k}$, $k=1,\ldots,n$. Here 
\begin{equation}\label{equ:10lagtan}
    T^{1,0}L_k:=\{s-iJs|\ s\in TL_k\}.
\end{equation}
The proposition below shows that there is a holomorphic bundle isomorphism transforms the loop in $\mathrm{U}(m)/\mathrm{O}(m)$ determined by the subbundles $F_k$ into diagonal form.

\begin{prop}\label{prop:standform}
    Let $E$ be a holomorphic vector bundle over $\D$ with fiber $\C^m$. Denote $I_k$, $k=1,\ldots,n$, as the closed intervals whose interiors are identified with the connected components of $\partial\D$ with finitely many points removed. Assume $I_k$ is ordered counterclockwisely. Suppose that $F_k$ is a Lagrangian subbundle of $E|_{I_k}$ over $I_k$ for each $k$. Then there is a map $\Theta:\D\rightarrow\mathrm{GL}(m,\C)$ holomorphic in the interior of $\D$ such that $\Theta^{-1}(F_k)$ is of the form
    \[
    \begin{bmatrix}
        f_k^1(\zeta)&0&\ldots&0\\
        0&f_k^2(\zeta)&\ldots&0\\
        ..&..&..&..\\
        0&0&\ldots&f_k^m(\zeta)
    \end{bmatrix}\R^m,\ \ \zeta\in I_k,
    \]
    in the Lagrangian Grassmannian space $\mathrm{U}(m)/\mathrm{O}(m)$.
\end{prop}

\begin{proof}
    Since $F_k$ is a Lagrangian subbundle of $E|_{I_k}$, then $F_k$ induces a path in the Lagrangian Grassmannian $\tau_k:I_k\rightarrow\mathrm{U}(m)/\mathrm{O}(m)$ with $k=1,\ldots,n$. As done in Definition \ref{def:msl2}, by multiplying $\tau_k$ with a family of diagonal matrix $\Phi_k\subset\mathrm{GL}(m,\C)$ parametrized by $I_k$, we get $n$ paths $v_k$ such that $\tau:=v_1\circ v_2\circ\ldots\circ v_n$ is a piecewise smooth loop in $\mathrm{U}(m)/\mathrm{O}(m)$. According to Corollary \ref{cor:ohsplitting}, there is a holomorphic map $\Theta:\D\rightarrow\mathrm{GL}(m,\C)$ such that 
    \[
    \tau(\zeta)=\Theta(\zeta)\Lambda^\frac12(\zeta)\R^m\ \ (\zeta\in\partial\D),
    \]
    where
    \[
    \Lambda(\zeta)=\begin{bmatrix}
        \zeta^{\kappa_1}&0&\ldots&0\\
        0&\zeta^{\kappa_2}&\ldots&0\\
        ..&..&..&..\\
        0&0&\ldots&\zeta^{\kappa_m}
    \end{bmatrix}.
    \]
    Therefore 
    \[
    (\Theta^{-1}\Phi_kF_k)(\zeta)=\Lambda^\frac12(\zeta)\R^m,\ \ \zeta\in I_k.
    \]
    Since $\Phi_k$ is a multiple of the identity matrix, $\Theta^{-1}\Phi_k=\Phi_k\Theta^{-1}$. This shows that $\Theta^{-1}(F_k)$ is of the desired form.
\end{proof}



The reminder of this section gives an equivalent description of the positivity of the partial Maslov indices. 

The first step is to deform the subbundles $F_k$ so that we can construct a holomorphic bundle over the 2-sphere out of $E$, where $E$ and $F_k$ are defined in the previous proposition. Since $F_k$ and $F_{k+1}$ do not agree at the endpoints, this issue can be fixed by multiplying a family of diagonal matrix $\Phi_k\subset\mathrm{GL}(m,\C)$ to $F_k$, as done in Definition \ref{def:msl2}. Denote the resulting deformed subbundles by $F'_k$, for $k=1,\ldots,n$. Then $F'_k$ can be patched together to get a Lagrangian subbundle $F$ of $E|_{\partial\D}$ that is piecewise smooth. By taking the conjugate bundle $\overline{E}$ over $\overline{\D}$, then $E$ and $\overline{E}$ can be glued along the boundary to obtain a bundle $E^D$ over $S^2$ ($F$ and $\overline{F}(={F})$ coincide). This bundle $E^D$ is holomorphic. In fact, using Proposition \ref{prop:standform}, the trivializations of $E^D$ are given by $\Theta^{-1}$ over $\D$, $\overline{\Theta}$ over $\overline{\D}$, and $\Lambda$ over a thickened neighborhood of $\partial\D$. This implies that $E^D$ is determined by the clutching function $\begin{bmatrix}
        \zeta^{\kappa_1}&0&\ldots&0\\
        0&\zeta^{\kappa_2}&\ldots&0\\
        ..&..&..&..\\
        0&0&\ldots&\zeta^{\kappa_m}
\end{bmatrix}$. Therefore $E^D$ is the holomorphic bundle $\oplus_k\mathcal{O}(\kappa_k)$. Here $\mathcal{O}(\kappa_k)$ stands for the degree $\kappa_k$ holomorphic line bundle over $S^2=\C\mathrm P^1$.\\

\begin{defn}\label{def:doubundle}
    Let $E$ and $F_k$ be as above with $k=1,\ldots,n$. The holomorphic bundle $E^D$ over $S^2$ constructed as above is called the \textbf{doubled holomorphic bundle}.
\end{defn}


The positivity of the partial Maslov indices requires the following concepts.

\begin{defn}\label{def:grfpos}
    Let $E$ be a hermitian holomorphic vector bundle over a complex K\"ahler manifold $X$ with the hermitian structure $h$. The \textbf{Chern connection} is the unique connection compatible with the hermitian structure and the holomorphic structure on $E$. Denote the $\mathrm{End}(E)$-valued $(1,1)$-form $\mathcal{K}$ as the corresponding \textbf{Chern curvature tensor}. The bundle $E$ is said to be \textbf{Griffiths positive} if at any $x\in X$,
    \[
    \langle\mathcal{K}\xi,\xi\rangle_h(s,\bar s)>0,\ \forall s\in T_xX\ and\ \xi\in E_x.
    \]
\end{defn}

The next proposition is important for Griffiths positivity.

\begin{prop}[{\cite[Proposition~6.10]{demailly1997complex}}]
    The quotient of a Griffiths positive hermitian vector bundle is Griffiths positive as well.
\end{prop}

\begin{thmbis}{thm:main2}
    Suppose $L_k$ with $k=1,\dots,n$ is a sequence of Lagrangians in a K\"ahler manifold $(X^m,\omega, J)$ intersecting transversely. Let $u:\D\rightarrow X$ be a disc connecting the intersections $\vec x,y$ of $\{L_k\}_k$. Let $E^D$ be the doubled holomorphic bundle defined as in Definition \ref{def:doubundle}, where $E=u^*T^{1,0}X$ and the boundary Lagrangian subbundle $F_k$ is $u^*T^{1,0}L_k|_{u(\partial\D)\cap L_k}$ $($Equation \ref{equ:10lagtan}$)$ for $k=1,\ldots,m$. Then the partial Maslov indices of $u$ are all larger than $1$ if and only if $E^D$ is Griffiths positive.
\end{thmbis}

\begin{proof}
    Let $\kappa_k$ be the partial Maslov indices of $u$, $k=1,\ldots,m$. By the definition of $E^D$, 
    \[
    E^D=\oplus_k\mathcal{O}(\kappa_k),
    \]
    where $\mathcal{O}(\kappa_k)$ is the degree $\kappa_k$ holomorphic line bundle over $S^2$. Therefore, it remains to show that $E^D$ is Griffiths positive if and only if $\kappa_k>0$ for all $k$.

    Suppose first that $E^D$ is Griffiths positive. The above proposition implies that each $\mathcal{O}(\kappa_k)$ is Griffiths positive. Therefore, $\mathcal{O}(\kappa_k)$ is ample and so $\kappa_k>0$, for all $k$.

    The converse is automatically true since the direct sum of Griffiths positive bundles is still Griffiths positive (because the Chern curvature tensor is diagonal after performing direct sum). 
\end{proof}

\section{The existence criterion of pseudo-holomorphic discs}\label{sec:3}

The aim of this section is to prove Theorem \ref{thm:main}. The method of proving the theorem is motivated by Chen-Fraser \cite{chen2015minimal} and Chen \cite{chen2025holomorphic}.

\begin{thmbis}{thm:main}
    Given a closed K\"ahler manifold $(X,\omega,J)$ and a sequence of Lagrangians $L_k$ with $k=1,\dots,n$ intersecting transversely in $X$, suppose there is a smooth disc $u$ connecting $\vec x=(x_1,\ldots,x_{n-1})$ and $y\in L_1\cap L_n$ with $x_k\in L_k\cap L_{k+1}$ for $k=1,\ldots,n-1$ $($Definition \ref{def:holodisc}$)$. Then there is a minimal disc $u_0$ connecting $\vec x$ and $y$ homotopic to $u$. If in addition, the partial Maslov indices $($Definition \ref{def:parmas}$)$ of $u_0$ are either all no less than $1$ or all no larger than $-1$, then there is a holomorphic map from $\D$ to $X$ with the same image as $u_0$.
\end{thmbis}

Before proving the above theorem, we need the following key property from \cite[Theorem~3.2]{chen2015minimal} and {\cite[Theorem~1.1(1)]{chen2015minimal}}.
 \begin{prop}[{\cite[Theorem~3.2]{chen2015minimal}} {\cite[Theorem~1.1(1)]{chen2015minimal}}]
    Let $L_k$, $k=1,\ldots,m$, be oriented Lagrangian submanifolds in an $m$-dimensional K\"ahler manifold $X$ and these Lagrangian submanifolds intersects each other transversely. Let $u_0:\D\rightarrow X$ be a minimal immersion such that $u_0(\partial\D)\subset \cup L_k$ and $u_0(\D)$ meets $\cup L_k$ orthogonally along $u_0(\D)$. Then if there exist $m$ 
    {holomorphic sections} of $E=u_0^*T^{1,0}X$ with boundary in $F_k:=u_0^*T^{1,0}L_k|_{u_0(\partial\D)\cup L_k}$ $($Equation \ref{equ:10lagtan}$)$ that are linear independent over $\R$ at some point $p\in\partial\D$ other than the intersections of $\{L_k\}_k$, then $u_0$ is holomorphic.
\end{prop}

\begin{rmk}
    \cite[Theorem~3.2]{chen2015minimal} and {\cite[Theorem~1.1(1)]{chen2015minimal}} proved the case where there is only one Lagrangian. In addition, the theorem proved in {\cite[Theorem~1.1(1)]{chen2015minimal}} is restricted to K\"ahler surfaces. \textbf{But the identical argument works here.}
\end{rmk}


    

\begin{proof}[Proof of Theorem \ref{thm:main}]
    The existence of the minimal disc is proved in Section \ref{sec:2}. It suffices to prove the second part.
    
    \textbf{Case 1:} all partial Maslov indices are no less than $1$.\\
    The goal of the proof is to find two $\R$-linear independent elements $s_1$ and $s_2$ in the kernel of the linearized operator of $\bar\partial_J$ with prescribed boundary conditions. If $s_1$ and $s_2$ are linearly independent over $\R$ at some point in $\partial \D$, then the proposition above guarantees the existence of a holomorphic disc. 
   
    According to Proposition \ref{prop:standform}, we can assume that the subbundle over $\partial\D$ is of the diagonal form as in Proposition \ref{prop:standform}. For $k$-th coordinate component, the Maslov index is the same as the corresponding partial Maslov index $\kappa_k$. Therefore, the dimension of the holomorphic section with boundary in $\pi_k(F_k)$ is $\kappa_k\ge1$, where $\pi_k$ is projection to the $k$-th component. Note that this space with holomorphic section is non-empty since it contains 0. Besides, any non-trivial section $s_k^0$ for each coordinate is non-constant because the Maslov index on is non-zero. As a result, there is a point $p\in\partial\D$ and $m$ real numbers $a_1,\ldots,a_m$ such that $s_k:=a_ks_k^0e_k|_p$ are $\R$-linear independent, where $e_k$ is the standard basis for $\R^m$. So the previous proposition implies that $u_0$ is holomorphic.

    \textbf{Case 2:} all partial Maslov indices are no lager than $-1$.\\
    In Case 1, $\D$ is in fact equipped with the standard complex structure with coordinate $z=s+it$. In the current case, the complex structure is $i$ on $\D$, but the coordinate is $w=\bar z=s-it$. Then the exact argument in the previous can be applied since in this coordinate, the partial Maslov indices are all no less than $1$ again.
\end{proof}



\begin{rmk}
    The results in this section still hold if the $n$th unit roots in the definition of holomorphic disc connecting $\vec x$ and $y$ are replaced by any other $n$ marked points.
\end{rmk}

We end the paper with proving Corollary \ref{cor:minidisccpm}, which states that all minimal discs in $\C\mathrm{P}^m$ with boundary in $\R\mathrm{P}^m$ are holomorphic.

\begin{proof}[Proof of Corollary \ref{cor:minidisccpm}]
    Let $u$ be a minimal discs in $\C\mathrm{P}^m$ with boundary in $\R\mathrm{P}^n$. Suppose $\tau$ is the anti-symplectic involution induced by taking the complex conjugation on $\C\mathrm{P}^m$. Then $\tau$ maps $u$ to another disc such that the new disc and $u$ can be glued together to get a $S^2$. Then the pullback bundle $E_1$ of $T^{1,0}\C\mathrm{P}^m\cong T\C\mathrm{P}^m$ to this $S^2$ is Griffiths positive since $T\C\mathrm{P}^m$ is Griffiths positive. Notice that the pullback of $T^{(0,1)}\C\mathrm{P}^m$ through $\tau\circ u:\overline{\D}\rightarrow \C\mathrm{P}^m$ is the conjugate bundle of $E:=u^*T^{(0,1)}\C\mathrm{P}^m$. Therefore, $E_1$ is the doubled holomorphic bundle $E^D$ $($Definition \ref{def:doubundle}$)$. After applying Theorem \ref{thm:main2} and Theorem \ref{thm:main}, $u$ must be holomorphic. Thus the corollary follows.
\end{proof}

\bibliographystyle{alpha}
\bibliography{references}

\noindent Qiang Tan,
   School of Mathematical Sciences,
  Jiangsu University, Zhenjiang, Jiangsu 212013, China\\
\noindent Email: tanqiang@ujs.edu.cn\\

\noindent Zuyi Zhang, 
Beijing International Center for Mathematical Research, 
 Peking University, Beijing, 100871, China\\
\noindent Email: zhangzuyi1993@pku.edu.cn\\
\noindent Personal page: sites.google.com/view/zuyizhangmath/home

\end{document}